\documentclass{amsart}
\usepackage{amssymb}
\usepackage{amsmath}
\usepackage{latexsym}
\usepackage{graphicx}
\usepackage{amscd}
\usepackage{eufrak}
\usepackage{mathrsfs}
\usepackage[english]{babel}
\usepackage{verbatim}

\newtheorem{theorem}[equation]{Theorem}
\newtheorem{theorem-definition}[equation]{Theorem-Definition}
\newtheorem{lemma-definition}[equation]{Lemma-Definition}
\newtheorem{definition-prop}[equation]{Proposition-Definition}
\newtheorem{corollary}[equation]{Corollary}
\newtheorem{prop}[equation]{Proposition}
\newtheorem{lemma}[equation]{Lemma}

\theoremstyle{definition}

\theoremstyle{definition}
\newtheorem{example}[equation]{Example}

\newtheorem{remark}[equation]{Remark}

\newcommand{\N}{\ensuremath{\mathbb{N}}}
\newcommand{\Z}{\ensuremath{\mathbb{Z}}}
\newcommand{\Q}{\ensuremath{\mathbb{Q}}}

\newcommand{\R}{\ensuremath{\mathbb{R}}}

\newcommand{\HH}{\ensuremath{\mathbb{H}}}

\newcommand{\cX}{\ensuremath{\mathscr{X}}}

\newcommand{\cC}{\ensuremath{\mathscr{C}}}
\newcommand{\cD}{\ensuremath{\mathscr{D}}}

\renewcommand{\R}{\ensuremath{\mathbb{R}}}

\newcommand{\Spec}{\ensuremath{\mathrm{Spec}\,}}

\newcommand{\red}{\mathrm{red}}

\newcommand{\lcm}{\mathrm{lcm}}

\newcommand{\wt}{\mathrm{wt}}

\newcommand{\Sk}{\mathrm{Sk}}

\newcommand{\an}{\mathrm{an}}

\numberwithin{equation}{subsection}

\newcommand{\sss}{\vspace{5pt} \subsubsection*{ }\refstepcounter{equation}{{\bfseries(\theequation)}\ }}

\begin{document}
\title{Weight functions on Berkovich curves}

\author[Matthew Baker]{Matthew Baker}
\address{School of Mathematics\\Georgia Institute of Technology\\Atlanta, GA 30332-0160\\USA} \email{mbaker@math.gatech.edu}

\author[Johannes Nicaise]{Johannes Nicaise}
\address{KU Leuven\\ Department of Mathematics\\ Celestijnenlaan 200B\\ 3001 Heverlee, Belgium
 \\
{\em Current address:} Imperial College\\
Department of Mathematics\\ South Kensington Campus \\
London SW7 2AZ, UK} \email{j.nicaise@imperial.ac.uk}

\begin{abstract}
 Let $C$ be a curve over a complete discretely valued field $K$. We give tropical descriptions of the weight function attached to a pluricanonical form on $C$ and the essential skeleton of $C$. We show that the Laplacian of the weight function equals the pluricanonical divisor on Berkovich skeleta, and we describe the essential skeleton of $C$
    as a combinatorial skeleton of the Berkovich skeleton of the minimal $snc$-model. In particular, if $C$ has semi-stable reduction, then the essential skeleton coincides with the minimal skeleton. As an intermediate step, we describe the base loci of logarithmic pluricanonical line bundles on minimal $snc$-models.
\end{abstract}

\maketitle

\section{Introduction}
 We denote by $R$ a complete discrete valuation ring with
quotient field $K$ and algebraically closed residue field $k$. Let $X$ be a smooth and proper $K$-variety. In \cite{MuNi}, Musta\c{t}\u{a} and the second author defined the {\em essential skeleton} $\Sk(X)$ of $X$, which is a finite simplicial complex embedded in the Berkovich analytification $X^{\an}$ of $X$. It is a union of faces of the Berkovich skeleton of any strict normal crossings model of $X$, but it does not depend on the choice of such a model.
 It was proven in \cite{NiXu} that, when $k$ has characteristic zero and the canonical line bundle on $X$ is semi-ample, the essential skeleton is a strong deformation retract of $X^{\an}$ and can be identified with the dual intersection complex of the special fiber of any minimal $dlt$-model of $X$ over $R$.
 The definition of the essential skeleton was based on the construction of a {\em weight function} $\wt_{\omega}$ on $X^{\an}$ attached to a pluricanonical form $\omega$ on $X$, which measures the degeneration of the pair $(X,\omega)$ locally at a point of $X^{\an}$. The aim of the present paper is to give an explicit description of the weight function and the essential skeleton in the case where $X$ is a curve, and to relate them to potential theory on graphs.

\medskip

 Let $C$ be a smooth, proper, geometrically connected curve over $K$. Denote by $\mathbb{H}_0(C)$ the Berkovich analytifcation $C^{\an}$ minus the points of type I and IV. In Section \ref{sec-metric}, we construct a metric on $\mathbb{H}_0(C)$ using the geometry of normal crossings models of $C$ over $R$. This is similar to the construction of the skeletal metric in the case where $K$ is algebraically closed \cite{BPR}, but our metric is not invariant under base change and cannot be obtained from the skeletal metric in any direct way. Using this metric, we can speak of integral affine functions on finite subgraphs of $C^{\an}$ and Laplacians of such functions.  Section \ref{sec-weight} is the heart of the paper: here we provide combinatorial descriptions of the weight function $\wt_{\omega}$ attached to a rational $m$-canonical form $\omega$ on $C$ and of the essential skeleton of $C$. Our first main result, Theorem \ref{thm:laplacian}, states that the Laplacian of the restriction of the weight function to the Berkovich skeleton of a suitable $snc$-model of $C$ equals the $m$-canonical divisor of the Berkovich skeleton, which is defined in terms of graph theory. Our second main result, Theorem \ref{thm:essentialsk}, states that the essential skeleton of a curve $C$ of positive genus is the subgraph of the Berkovich skeleton of the minimal $snc$-model of $C$ obtained by contracting all the tails of rational curves. In particular, if $C$ has semi-stable reduction, then the essential skeleton of $C$ is equal to the Berkovich skeleton of its minimal $snc$-model. The proof of Theorem \ref{thm:essentialsk} is based on Theorem \ref{thm:semiample}, which describes the base locus of the logarithmic relative pluricanonical bundles of the minimal $snc$-model of $C$. We also prove that, in the semi-stable reduction case, it suffices to look at weight functions of $2$-canonical forms to recover the essential skeleton; moreover, if the  essential skeleton of $C$ is bridgeless, then canonical forms suffice (see Theorem \ref{thm:nonbridgeessentialsk}).
 Finally, in Section \ref{sec:appendix}, we describe a different natural metric on $\mathbb{H}_0(C)$ which behaves better under (tame) base change and which is closer to the skeletal metric from \cite{BPR}.

\subsection{Acknowledgements}
 The authors would like to thank Mattias Jonnson for helpful comments on an earlier version of this text.
 The second author was supported by the ERC Starting Grant MOTZETA (project 306610) of the European Research Council.

\subsection{Notation}
\sss We denote by $R$ a complete discrete valuation ring with
quotient field $K$ and algebraically closed residue field $k$. We assume that the valuation $v_K$ on $K$ is normalized, i.e., that $v_K(t)=1$ for any uniformizer $t$ in $R$, and we define an absolute value $|\cdot|_K$ on $K$ by setting $|a|_K=\exp(-v_K(a))$ for every $a$ in $K^{\ast}$. We fix an algebraic closure $K^a$ of $K$. The absolute value $|\cdot|_K$ extends uniquely to an absolute value on $K^a$, which we still denote by $|\cdot|_K$. We write $\widehat{K^a}$ for the completion of $K^a$ with respect to $|\cdot|_K$.

\sss By a curve over $K$, we will mean a geometrically connected smooth proper $K$-variety of dimension one. For every scheme $S$ we denote by $S_{\red}$ the maximal reduced closed subscheme. For every $R$-scheme $\cX$ we set $\cX_K=\cX\times_R K$ and $\cX_k=\cX\times_R k$. If $\mathcal{L}$ is a line bundle on a scheme $X$ and $D$ is a Cartier divisor on $X$, then we write $\mathcal{L}(D)$ for the line bundle $\mathcal{L}\otimes\mathcal{O}_X(D)$, as usual.

\sss We will work with the category of $K$-analytic spaces as defined by Berkovich \cite{berkbook}. We assume a basic familiarity with the theory of analytic curves over $K$; see for instance \cite{BPR}.

\section{The metric on the Berkovich analytification of a $K$-curve}\label{sec-metric}
\subsection{Metric graphs associated to curves with normal crossings}
\sss When we speak of a discrete graph $G$, we mean a finite connected undirected
 multigraph, i.e., we allow multiple loops and multiple edges between vertices.
   We denote the vertex set of $G$ by $V(G)$ and the set of
edges by $E(G)$. A weighted discrete graph is a couple $(G,w)$ where $G$ is a discrete graph and $w$ is a function
$$w:V(G)\to \R.$$

\sss A discrete graph $G$ has a geometric realization $\Gamma$,
which is a defined as follows: we start from the set $V(G)$
 and we attach one copy of the closed interval $[0,1]$ between two
vertices $v_1$ and $v_2$ for each edge of $G$ with endpoints $\{v_1,v_2\}$.
 If $G$ is endowed with a weight function $w$ that takes values in $\Z_{>0}$, then we can turn the topological space
 $\Gamma$ into a metric space by declaring that the length of every edge $e$ between two adjacent vertices
 $v_1$ and $v_2$ is equal to
 \begin{equation}\label{def:metric}
 \ell(e)=\frac{1}{w(v_1)\cdot w(v_2)}.
 \end{equation}
 In these definitions, we allow the possibility that $v_1=v_2$.
 We call the metric space $\Gamma$ the metric graph associated with $(G,w)$.

\sss Let $X$ be a connected separated $k$-scheme of finite type, of
pure dimension one. We say that $X$ has normal crossings if the
 only
singular points of $X_{\red}$ are ordinary
double points. We associate a weighted discrete graph $(G(X),w)$ to $X$ as
 follows.
 The vertex set of $G(X)$ is the set of irreducible
components of $X$ and the edge set of $G(X)$
is the set of singular points of $X_{\red}$. If $e$ is an edge
corresponding to a singular point $x$ of $X_{\red}$, then the end points
of $e$ are the vertices corresponding to the irreducible
components of $X$ containing $x$. In particular, $e$ is a loop if
and only if $x$ is a singular point of an irreducible component of
$X$. If $v$ is a vertex of $G(X)$ corresponding to an irreducible component $E$ of $X$, then
 the weight $w(v)$ is defined to be the multiplicity of $X$ along $E$, i.e., the length of the local ring of $X$ at the generic point of $E$. The metric graph associated with
 $(G(X),w)$ will be denoted by $\Gamma(X)$.

\subsection{Models with normal crossings}
\sss Let $C$ be a curve over $K$.
 An $nc$-model of $C$ is a regular flat proper $R$-scheme $\cC$, endowed with an isomorphism of $K$-schemes $\cC_K\to C$, such that the special fiber
 $\cC_k$ has normal crossings. We call $\cC$ an $snc$-model of $C$ if, moreover, $\cC_k$ has strict normal crossings, which means that its irreducible components (endowed with the induced reduced structure) are regular.  If $\cC$ and $\cC'$ are $nc$-models of $C$, then a morphism of $R$-schemes $h:\cC'\to \cC$ is called a morphism of $nc$-models if the morphism $h_K:\cC'_K\to \cC_K$ obtained by base change to $K$ commutes with the isomorphisms to $C$. Morphisms of $snc$-models are defined analogously.
  We say that $\cC'$ dominates $\cC$ if there exists a morphism of $nc$-models $\cC'\to \cC$; such a morphism is automatically unique. We denote this property by $\cC'\geq \cC$. The relation $\geq$ defines a partial ordering on the set of isomorphism classes of $nc$-models of $C$. This partial ordering
  is filtered, and the $snc$-models form a cofinal subset since any $nc$-model can be transformed into an $snc$-model by blowing up at the self-intersection points of the irreducible components of the special fiber.
 We say that the curve $C$ has semi-stable reduction if any relatively minimal $nc$-model of $C$ has a reduced special fiber. Beware that this does not imply that the minimal $snc$-model has reduced special fiber, as blowing up at self-intersection points introduces components of multiplicity two.

\sss
Denote by $C^{\an}$ the Berkovich analytification of $C$, and let $\cC$ be an $snc$-model of $C$. If $E$ is an irreducible component of $\cC_k$ and $v$ denotes the corresponding vertex of the weighted discrete graph $(G(\cC_k),w)$, then $w(v)$ is precisely the multiplicity of $E$ in the divisor $\cC_k$.  In \cite[\S3.1]{MuNi}, Musta\c{t}\u{a} and the second author
 defined a canonical topological embedding of the metric graph $\Gamma(\cC_k)$ into $C^{\an}$, generalizing a construction by Berkovich. The image of this embedding is called the Berkovich skeleton
of the model $\cC$ and denoted by $\Sk(\cC)$. By \cite[3.1.5]{MuNi}, the embedding of $\Sk(\cC)$ into $C^{\an}$ has a canonical continuous retraction $$\rho_{\cC}:C^{\an}\to \Sk(\cC).$$ If we let $\cC$ vary over the class of $snc$-models of $C$, ordered by the domination relation, then the maps $\rho_{\cC}$ induce a homeomorphism $$C^{\an}\to \lim_{\stackrel{\longleftarrow}{\cC}}\Sk(\cC).$$
 This is easily proven by an adaptation of the argument in \cite[5.2]{BPR} (where it is assumed that the base field is algebraically closed).
 It is straightforward to generalize these constructions to $nc$-models, either by copying the arguments or by observing
 that blowing up $\cC$ at all the self-intersection points of irreducible components of $\cC_k$, we get an $snc$-model $\cC'$ of $C$ and the morphism $\cC'\to \cC$
 induces an isometry $\Gamma(\cC'_k)\to \Gamma(\cC_k)$ (the effect of this operation on $\Gamma(\cC_k)$ is that we add a vertex in the middle of every loop).

\subsection{Definition of the metric}\label{ss:def-metric}
\sss Let $C$ be a curve over $K$, and denote by $\HH_0(C)$
  the subset of $C^{\an}$ obtained by removing the points of type I and IV.

\begin{lemma}\label{lemm:union}
  For every $nc$-model $\cC$ of $C$, the Berkovich skeleton $\Sk(\cC)$ is contained in $\HH_0(C)$. Moreover, $\HH_0(C)$ is the union of the
  skeleta $\Sk(\cC)$ where $\cC$ runs through any cofinal set of $nc$-models of $\cC$.
\end{lemma}
\begin{proof}
The first part of the statement is obvious from the construction of $\Sk(\cC)$. The second
 part follows from the fact that this union is connected and contains all type II points of $C^{\an}$, by \cite[2.4.11 and 3.1.7]{MuNi}.
\end{proof}

 The following theorem explains how one can define a natural
 metric on the set $\HH_0(C)$.

\begin{theorem}\label{thm:metric}
There exists a unique metric on $\HH_0(C)$ such that, for every $nc$-model $\cC$ of $C$, the map
$$\Gamma(\cC_k)\to \HH_0(C)$$ is an isometric embedding.
\end{theorem}
\begin{proof}
The uniqueness of the metric is obvious from Lemma \ref{lemm:union}. Thus it suffices to prove its existence.  Let $\cC$ and $\cC'$ be $nc$-models of $C$ such that $\cC'$ dominates $\cC$.
 Then the skeleton $\Sk(\cC)$ is contained in $\Sk(\cC')$ by \cite[3.1.7]{MuNi}, and it suffices to show that the corresponding embedding
 $\Gamma(\cC_k)\to \Gamma(\cC'_k)$ is an isometry. Since we can decompose the morphism $\cC'\to \cC$ into a finite composition of point blow-ups, we can assume that $\cC'\to \cC$
 is the blow-up of $\cC$ at a closed point $x$ of $\cC_k$. If $x$ is a regular point of $(\cC_k)_{\red}$ then the claim is obvious. If $x$ is a singular point
  then $(G(\cC'_k),w)$ is obtained from $(G(\cC_k),w)$ by adding a vertex on the edge $e$ corresponding to $x$ and giving it weight $w(v_1)+w(v_2)$, where $v_1$ and $v_2$
   are the (not necessarily distinct) endpoints of $e$.
    The lengths of the segment $e$ in the metric graphs $\Gamma(\cC_k)$ and $\Gamma(\cC'_k)$ are the same, because
    $$\frac{1}{w(v_1)\cdot w(v_2)}=\frac{1}{w(v_1)\cdot(w(v_1)+w(v_2))}+\frac{1}{(w(v_1)+w(v_2))\cdot w(v_2)}.$$
\end{proof}

\begin{remark}
 There is another metric on $\HH_0(C)$ that is induced by the piecewise integral affine structure on the skeleta of $snc$-models; we will explain its construction in Section \ref{sec:appendix}. Although this second metric arises more naturally and behaves better under base change, the one we defined in Theorem \ref{thm:metric} seems to be the correct one for the purposes of potential theory. A similar discrepancy appears in the non-archimedean study of germs of algebraic surfaces, which is in many ways analogous to the set-up we consider here; see Section 7.4.10 of \cite{jonsson}.
  \end{remark}

%
%

\section{The weight function and the essential skeleton}\label{sec-weight}
\subsection{The weight function attached to a pluricanonical form}
\sss \label{sss:wtnot} We fix a $K$-curve $C$. Let $m$ be a positive integer and let $\omega$ be a non-zero rational $m$-canonical form on $C$. Thus $\omega$ is a non-zero rational section of the $m$-canonical line bundle $\omega_{C/K}^{\otimes m}$. As such, it defines a Cartier divisor on $C$, which we denote by $\mathrm{div}_C(\omega)$.
 If $\cC$ is any $snc$-model of $C$, we can also view $\omega$ as a rational section of the logarithmic relative $m$-canonical line bundle $$\omega_{\cC/R}(\cC_{k,\red})^{\otimes m}$$ and we denote the corresponding divisor on $\cC$ by $\mathrm{div}_{\cC}(\omega)$. Note that the horizontal part of $\mathrm{div}_{\cC}(\omega)$ is simply the schematic closure of $\mathrm{div}_C(\omega)$ in $\cC$.

\sss In \cite[4.4.4]{MuNi}, Musta\c{t}\u{a} and the second author attached to $\omega$ a so-called {\em weight function} $\wt_{\omega}$. In our setting (the case of curves) we can characterize its restriction to $\mathbb{H}_0(C)$ in the following way. Note that the points of type II on $C^{\an}$ are precisely the divisorial points in the sense of \cite[2.4.10]{MuNi}, and that the points of type II and III are precisely the monomial points.

\begin{prop}\label{prop:defwt}
The weight function
$$\wt_{\omega}:\mathbb{H}_0(C)\to \R$$ is the unique function with the following properties, for every $snc$-model $\cC$ of $C$:
\begin{enumerate}
\item The restriction of $\wt_{\omega}$ to $\Sk(\cC)$ is continuous with respect to the metric topology (which coincides with the Berkovich topology on $\Sk(\cC)$).
\item Let $E$ be an irreducible component of $\cC_k$.  We denote by $N$ and $\nu$ the multiplicities of $E$ in $\cC_k$ and $\mathrm{div}_{\cC}(\omega)$, respectively. If $x$ is the divisorial point of $C^{\an}$ attached to $(\cC,E)$ (equivalently, the vertex of $\Sk(\cC)$ corresponding to $E$), then
    $$\wt_{\omega}(x)=\frac{\nu}{N}.$$
    \end{enumerate}
\end{prop}
\begin{proof}
 It is shown in \cite[4.3.3]{MuNi} that the weight function is continuous (even piecewise affine) on $\Sk(\cC)$, and the description at divisorial points is part of its definition.
 Uniqueness is clear from Lemma \ref{lemm:union} and the fact that the divisorial points are dense in the skeleton of every $snc$-model of $C$ (by the proof of  \cite[2.4.11]{MuNi}, they correspond precisely to the points on $\Gamma(\cC)$ with rational barycentric coordinates in the sense of \cite[3.1.2]{MuNi}).
\end{proof}

\sss Beware that the weight function is not continuous with respect to the Berkovich topology on $\mathbb{H}_0(C)$ (see \cite[4.4.6]{MuNi} for a counterexample).
  The explicit description of the weight function given in Theorem \ref{thm:laplacian} below shows in particular that it is continuous with respect to the metric topology on $\mathbb{H}_0(C)$ (which is strictly finer than the Berkovich topology).

\subsection{The Laplacian of the weight function}
\sss By a {\em pair} over $K$, we mean a couple $(C,\delta)$ consisting of a $K$-curve $C$
 and a divisor $\delta$ on $C$. An $nc$-model of a pair $(C,\delta)$ is an $nc$-model $\cC$ of $C$ such that the sum of $\cC_k$ with the schematic closure of $\delta$ is a normal crossings divisor on $\cC$. An $snc$-model of $(C,\delta)$ is defined analogously.
  Note that for every point $x$ in the support of $\delta$, the specialization of $x$ to $\cC_k$ lies in a unique irreducible component $E$ of $\cC_k$, and the multiplicity of $\cC_k$ along $E$ is equal to the degree of $x$ over $K$, by the normal crossings condition.
   The skeleton of $(\cC,\delta)$ is defined to be the intersection of $\mathbb{H}_0(C)$ with the convex hull in $C^{\an}$ of $\Sk(\cC)$ and the support of $\delta$. We will denote it by $\Sk(\cC,\delta)$. Thus we obtain $\Sk(\cC,\delta)$ from $\Sk(\cC)$ by adding, for each point $x$ in the support of $\delta$, the open branch running from $\Sk(\cC)$ towards $x$.  This construction is similar to the definition of the skeleton of a strictly semi-stable pair in \cite{GubRabWer}, but there it is assumed that $K$ is algebraically closed and that $\cC_k$ is reduced and has strict normal crossings. By restricting the metric on $\mathbb{H}_0(C)$ to the skeleton $\Sk(\cC,\delta)$, we can view the skeleton as a metric graph with some half-open edges of infinite length. Then it makes sense to speak about a $\Z$-affine function $f$ on $\Sk(\cC,\delta)$ (i.e., a continuous real-valued function that is integral affine on every edge) and the Laplacian $\Delta(f)$ of such a function (the divisor on $\Sk(\cC,\delta)$ whose degree at a vertex is the sum of the outgoing slopes of $f$).

\sss Our aim is to give a combinatorial description of the weight function $\wt_{\omega}$ on $\mathbb{H}_0(C)$ attached to a non-zero rational $m$-canonical form $\omega$ on $C$. For this description we need to introduce the $m$-canonical divisor of a labelled graph. Let $G$ be a discrete graph without loops, where we allow some of the edges of $G$ to be half-open (i.e., the edge has only one adjacent vertex and is unbounded at the other side). Assume that each vertex $v$ of $G$ is labelled by a couple of non-negative integers $(N(v),g(v))$. Then the canonical divisor of $G$ is defined by
$$K_G=\sum_{v\in V(G)}N(v)(\mathrm{val}(v)+2g(v)-2)v$$
where $\mathrm{val}(v)$ denotes the valency at $v$, that is, the number of edges (bounded and unbounded) in $G$ adjacent to $v$.
 When $N(v)=1$ and $g(v)=0$ for every vertex $v$, this is just the usual definition of the canonical divisor of a discrete graph.
The $m$-canonical divisor of $G$ is defined as $m$ times the canonical divisor $K_G$.

 \begin{theorem}\label{thm:laplacian}
  We fix a $K$-curve $C$. Let $m$ be a positive integer and let $\omega$ be a non-zero rational $m$-canonical form on $C$.
 Let $\delta$ be any divisor on $C$ whose support contains the support of $\mathrm{div}_C(\omega)$ and let $\cC$ be an $snc$-model for the pair $(C, \delta)$.
 \begin{enumerate}
 \item \label{it:affine} The weight function  $\wt_{\omega}$ is $\Z$-affine on every edge of $\Sk(\cC,\delta)$.
 \item \label{it:slope} For every point $x$ in the support of $\delta$, the weight function $\wt_{\omega}$ has constant slope
 on the path running from $\Sk(\cC)$ to $x$ in $C^{\an}$, and this slope is equal to $$N(m+\mathrm{deg}_x(\mathrm{div}_C(\omega)))$$ where $N$ denotes the multiplicity of the unique component in $\cC_k$ containing the specialization of $x$.
 \item \label{it:lapl} The Laplacian of the restriction of $\wt_{\omega}$ to $\Sk(\cC,\delta)$ is equal to the $m$-canonical divisor of the graph $\Sk(\cC,\delta)$ if we label each vertex $v$ with the couple $(N(v),g(v))$ where $N(v)$ is
      the multiplicity of the corresponding irreducible component in $\cC_k$ and $g(v)$ denotes its genus.
 \end{enumerate}
  \end{theorem}
\begin{proof}
\eqref{it:affine} It follows from the proof of \cite[4.3.3]{MuNi} that $\wt_{\omega}$ is $\Z$-affine on $\Sk(\cC)$, because no point in the support of $\mathrm{div}_C(\omega)$ specializes to a singular point of $(\cC_k)_{\red}$. Here some care is needed, since the $\Z$-affine structure in \cite{MuNi} is not the same as the one induced by our metric; it corresponds to the metric one obtains by replacing the definition in \eqref{def:metric} by $$\ell(e)=\frac{1}{\lcm\{w(v_1),w(v_2)\}}.$$ Since this multiplies every edge length by an integer factor, every $\Z$-affine function in the sense of \cite{MuNi} is also $\Z$-affine with respect to the metric we use; we will come back to this point in Section \ref{sec:appendix}. The fact that $\wt_{\omega}$ is also $\Z$-affine on the unbounded edges of $\Sk(\cC,\delta)$ is a consequence of \eqref{it:slope}.

\eqref{it:slope} Let $x$ be a closed point of $C$.
 We can compute the slope of $\wt_{\omega}$ on the path running from $\Sk(\cC)$ to $x$ as follows. Denote by $E$ the unique irreducible component of $\cC_k$ containing the specialization $x_k$ of $x$.
  Denote by $N$ the multiplicity of $E$ in $\cC_k$ and by $\nu$ the multiplicity of $E$ in $\mathrm{div}_{\cC}(\omega)$. Let $h:\cC'\to \cC$ be the blow-up at $x_k$. Then $\cC'$ is again an $snc$-model of $(C,\delta)$ and its skeleton $\Sk(\cC')$ is obtained from $\Sk(\cC)$ by adding a closed interval $I$ in the direction of $x$. The length of this interval is $1/N^2$, since the exceptional component $E'$ of the blow-up has multiplicity $N$ in $\cC'_k$. Moreover, the multiplicity of $E'$ in $\mathrm{div}_{\cC'}(\omega)$ is equal to
    $$\nu+m+\mathrm{deg}_x(\mathrm{div}_C(\omega)))$$
    because
    $$\omega^{\otimes m}_{\cC'/R}(\cC'_{k,\red})=(h^*\omega^{\otimes m}_{\cC/R}(\cC_{k,\red}))\otimes \mathcal{O}_{\cC'}(-mE')$$ as submodules of the pushforward of $\omega^{\otimes m}_{C/K}$ to $\cC'$.
 Thus if we denote by $v$ and $v'$ the vertices of $\Sk(\cC')$ corresponding to $E$ and $E'$, respectively, then $\wt_{\omega}(v)=\nu/N$ and  $$\wt_{\omega}(v')=(m+\nu+\mathrm{deg}_x(\mathrm{div}_C(\omega)))/N.$$ Since $v$ and $v'$ are precisely the endpoints of $I$, we see that $\wt_{\omega}$ has slope
  $$N(m+\mathrm{deg}_x(\mathrm{div}_C(\omega)))$$ on $I$ if we orient $I$ from $v$ to $v'$. Replacing $\cC$ by $\cC'$ and repeating the argument, we conclude that $\wt_{\omega}$ has constant slope
   $$N(m+\mathrm{deg}_x(\mathrm{div}_C(\omega)))$$ along the whole path from $v$ to $x$.

 \eqref{it:lapl}  It remains to compute the Laplacian $\Delta(\wt_{\omega})$ of $\wt_{\omega}$ on $\Sk(\cC,\delta)$. Let $v_0$
    be a vertex of $\Sk(\cC)$ corresponding to an irreducible component $E_0$ of $\cC_k$. Denote by $x_1,\ldots,x_a$ the
     points in the support of $\delta$ that specialize to a point in $E_0$, and by $y_1,\ldots,y_b$ the intersection points of $E_0$ with the other irreducible components of $\cC_k$. For each $i\in \{1,\ldots,b\}$ we denote by $E_i$ the unique irreducible component of $\cC_k$ intersecting $E_0$ at $y_i$. For each $i\in \{0,\ldots,b\}$ we write $\nu_i$ and $N_i$ for the multiplicities of $E_i$ in $\mathrm{div}_{\cC}(\omega)$ and $\cC_k$, respectively. Then the edges of $\Sk(\cC)$ adjacent to $v_0$ correspond precisely to the points $y_1,\ldots,y_b$, and the unbounded edges of $\Sk(\cC,\delta)$ adjacent to $v_0$ are precisely the paths from $v_0$ to the points $x_1,\ldots,x_a$. We have already computed the slopes of $\wt_{\omega}$ along these unbounded edges, and taking into account the edge lengths of $\Sk(\cC)$ we find that the degree of $\Delta(\wt_\omega)$ at $v_0$ is equal to
     $$\sum_{i=1}^a (N_0m+\mathrm{deg}_{x_i}(\mathrm{div}_C(\omega))) +\sum_{j=1}^b(\nu_jN_0-\nu_0N_j).$$
This is nothing but
$$mN_0a+N_0(E_0\cdot (\mathrm{div}_{\cC}(\omega)-\frac{\nu_0}{N_0}\cC_k))=mN_0a+N_0(E_0\cdot \mathrm{div}_{\cC}(\omega)). $$
By adjunction, the restriction of the line bundle $\omega_{\cC/R}(\cC_{k,\red})^{\otimes m}$ to $E_0$ is precisely
$$\omega_{E_0/k}(y_1+\ldots+y_b)^{\otimes m}.$$
By computing the degree of this line bundle we find that the degree of $\Delta(\wt_\omega)$ at $v_0$ is equal to
$$mN_0(a+b+2g(E_0)-2)$$
where $g(E_0)$ denotes the genus of $E_0$. By definition, this is exactly the degree of the  $m$-canonical divisor of $\Sk(\cC,\delta)$ at $v_0$.
\end{proof}

\sss  We can use Theorem \ref{thm:laplacian} to describe the Laplacian of the restriction of the weight function to the skeleton of {\em any} $snc$-model $\cC$ of $C$.
   Beware that the weight function is not necessarily affine on the edges of $\Sk(\cC)$, only piecewise affine. The Laplacian of such a function is still defined, but it is no longer supported on the vertices of $\Sk(\cC)$, in general.
 We denote by $(\rho_{\cC})_*$ the map on divisors induced by linearity from the retraction map $\rho_{\cC}: C^\an \to \Sk(\cC)$.
  We have $$(\rho_{\cC})_*(x)=\mathrm{deg}(x)\cdot \rho_{\cC}(x)$$ for every type I point $x$ of $C^{\an}$.

\begin{corollary}\label{cor:bound}
Let $\cC$ be any $snc$-model of $C$. We denote by $f$
 the restriction of $\wt_{\omega}$ to $\Sk(\cC)$ and by $mK_{\Sk(\cC)}$ the $m$-canonical divisor of $\Sk(\cC)$, where we label each vertex of $\Sk(\cC)$ by its multiplicity and genus as before. Then $$\Delta(f) = mK_{\Sk(\cC)} - (\rho_{\cC})_* \left( {\rm div}_C(\omega) \right).$$ In particular, if $\omega$ is regular, then $\Delta(f) \leq mK_{\Sk(\cC)}$.
\end{corollary}
\begin{proof}
 We can always dominate $\cC$ by an $snc$-model $\cC'$ of the pair $(C,\mathrm{div}_C(\omega))$. If we denote by $f'$ the restriction of $\wt_{\omega}$ to $\Sk(\cC')$, then it follows easily from Theorem \ref{thm:laplacian} that
 $$\Delta(f') = mK_{\Sk(\cC')} - (\rho_{\cC'})_* \left( {\rm div}_C(\omega) \right).$$ Denote by $\rho:\Sk(\cC')\to \Sk(\cC)$ the map of metric graphs obtained by restricting $\rho_{\cC}$ to $\Sk(\cC')$. Since the fibers of $\rho$ are metric trees, it is straightforward to check that $\Delta(f) = \rho_* (\Delta(f'))$. On the other hand, we also have that $$\rho_*(\rho_{\cC'})_* \left( {\rm div}_C(\omega) \right)=(\rho_{\cC})_* \left( {\rm div}_C(\omega) \right),$$ and by factoring $\cC'\to \cC$ into point blow-ups, one sees that $\rho_*(K_{\Sk(\cC')})=K_{\Sk(\cC)}$. Thus the formula is valid for $\cC$, as well.
\end{proof}

\begin{example} \label{sss:exam} Let $C$ be an elliptic curve over $K$ of Kodaira-N\'eron reduction type II and let $\omega$ be a generator for the relative canonical line bundle of the minimal regular model of $C$. Let $\cC$ be the minimal $snc$-model of $C$. Then the special fiber of $\cC$  is of the form $$\cC_k=E_1+2E_2+3E_3+6E_4$$ where each component $E_i$ is a rational curve, $E_4$ intersects each other component in precisely one point and there are no other intersection points. The skeleton $\Sk(\cC)$ consists of four vertices $v_1,\ldots,v_4$ corresponding to the components $E_1,\ldots,E_4$. These are joined by three edges of respective lengths $\ell(v_1v_4)=1/6$, $\ell(v_2v_4)=1/12$ and $\ell(v_3v_4)=1/18$. Moreover,
$$\mathrm{div}_{\cC}(\omega)=E_1+2E_2+3E_3+5E_4$$ and the weight function $\wt_{\omega}$ is affine on $\Sk(\cC)$ with values $1,1,1,5/6$ at the vertices $v_1,v_2,v_3,v_4$, respectively. Direct computation shows that $$\Delta\wt_{\omega}=6v_4-v_1-2v_2-3v_3$$ which is also the canonical divisor of $\Sk(\cC)$ (labelled with multiplicities and genera).
\end{example}

\begin{remark} \label{rmk:nondiscretecase}
It is worth noting that Corollary~\ref{cor:bound} uniquely determines $\wt_\omega$ up to an additive constant as a function on $\mathbb{H}_0(C)$, and that this description of $\wt_\omega$ does not require $K$ to be discretely valued (if we replace $snc$-models by semi-stable models).  This gives us a way to define $\wt_\omega$ for any curve $C$ over any non-trivially valued
non-Archimedean field $K$ and any non-zero rational $m$-canonical form $\omega$ on $C$.  Michael Temkin \cite{Temkin} has recently discovered a different way to extend the definition of $\wt_\omega$ to the non-discretely valued setting, and his method works in any dimension.
\end{remark}

\subsection{The essential skeleton}
\sss \label{sss:KSsk0} Let $C$ be a $K$-curve of genus $g(C)\geq 1$ and let $\omega$ be a non-zero regular $m$-canonical form on $C$, for some positive integer $m$. Then it is easy to deduce
 from the properties of the weight function $\wt_{\omega}$ in Theorem \ref{thm:laplacian} that this function is bounded below, and that its locus
 of minimal values is a union of closed faces of $\Sk(\cC)$ for any $snc$-model $\cC$ of $(C,\mathrm{div}_C(\omega))$. Corollary \ref{cor:bound} shows that this remains true for any $snc$-model $\cC$ of $C$ (one needs to observe that the weight function is {\em concave} on every edge of $\Sk(\cC)$ because its Laplacian is non-positive at each point in the interior of an edge);
  see \cite[4.5.5]{MuNi} for a more general statement.
 The locus of minimal values of $\wt_\omega$ was called the Kontsevich-Soibelman skeleton of the pair $(C,\omega)$ in \cite[4.5.1]{MuNi} and denoted by $\Sk(C,\omega)$. The essential skeleton $\Sk(C)$ is the union of the Kontsevich-Soibelman skeleta $\Sk(C,\omega)$ over all the non-zero regular pluricanonical forms $\omega$ on $C$ (see \cite[4.6.2]{MuNi}). The aim of this section is to compare
 the essential skeleton $\Sk(C)$ to the skeleton $\Sk(\cC)$ of the minimal $snc$-model $\cC$ of $C$.

\sss \label{sss:KSsk} We first recall the description of the Kontsevich-Soibelman skeleton $\Sk(C,\omega)$ in terms of an $snc$-model $\cC$ of $C$ (see \cite[4.5.5]{MuNi} for a more general result; in our setting, it can also be easily deduced from Proposition \ref{prop:defwt} and Theorem \ref{thm:laplacian}). We write $$\cC_{k}=\sum_{i\in I}N_i E_i$$ and we denote by $\nu_i$ the multiplicity of $E_i$ in $\mathrm{div}_{\cC}(\omega)$, for every $i\in I$. We say that the vertex of $\Sk(\cC)$ corresponding to a component $E_j$, $ j\in I$, is {\em $\omega$-essential} if
$$\frac{\nu_j}{N_j}=\min\{\frac{\nu_i}{N_i}\,|\,i\in I\}.$$ We say that an edge in $\Sk(\cC)$ is $\omega$-essential if its adjacent vertices are $\omega$-essential and the point of $\cC_k$ corresponding to the edge is not contained in the closure of $\mathrm{div}_C(\omega)$ (i.e., the horizontal part of $\mathrm{div}_{\cC}(\omega)$). Then $\Sk(C,\omega)$ is the union of the $\omega$-essential faces of $\Sk(\cC)$. Note however that, by its very definition, $\Sk(C,\omega)$ does not depend on the choice of a particular model $\cC$.

\sss \label{sss:tails} In order to determine the essential skeleton $\Sk(C)$, we will need a description of the base locus of the logarithmic pluricanonical bundle on the minimal $snc$-model of $C$. Let $\cC$ be any $snc$-model of $C$. We  label the vertices of the skeleton $\Sk(\cC)$ by the multiplicities and genera of the corresponding irreducible components of $\cC_k$.
 We define a {\em tail} in $\Sk(\cC)$ as a connected subchain with successive vertices $v_0,\ldots,v_n$ where $v_n$ has valency one in $\Sk(\cC)$, $v_i$ has valency $2$ in $\Sk(\cC)$ for $1\leq i<n$, and $v_i$ has genus zero for $1\leq i\leq n$. We say that the tail is {\em maximal} if $v_0$ has valency at least $3$ in $\Sk(\cC)$ or $v_0$ has positive genus. The vertex $v_0$ is called the starting point of the maximal tail and $v_n$ is called its end point. We call the components of $\cC_k$ corresponding to the vertices $v_1,\ldots,v_n$ {\em inessential} components of $\cC_k$.
 The {\em combinatorial skeleton} of $\Sk(\cC)$ is the subspace that we obtain by replacing every maximal tail by its starting point.
 Thus in Example \ref{sss:exam}, the combinatorial skeleton of $\Sk(\cC)$ consists only of the vertex $v_4$. Note that contracting maximal tails may create new ones, but we do {\em not} repeat the operation to contract those. For instance, if $C$ is an elliptic $K$-curve of reduction type $I_n^*$ and $\cC$ is its minimal $snc$-model, then the combinatorial skeleton of $\Sk(\cC)$ is the subchain formed by the $n+1$ vertices of multiplicity two. Note that $\cC_k$ can never consist entirely of inessential components, by our assumption that $g(C)\geq 1$ (this follows from basic intersection theory and adjunction; see for instance \cite[3.1.2]{Ni-tameram}). We also observe that, if $C$ has semi-stable reduction and $\cC$ is its minimal $snc$-model, there are no inessential components in $\cC_k$ because the end point of a tail would correspond to a rational $(-1)$-curve, which contradicts the minimality of $\cC$.

\sss We will need a technical lemma on two-dimensional log regular schemes. We refer to \cite{kato-log} for the basic theory of log schemes, and to \cite{kato-toric}
for the theory of log regular schemes.

\begin{lemma}\label{lemm:factorial}
Let $A$ be a normal Noetherian local ring of dimension $2$ and let $D=D_0+D_1$ be a reduced Weil divisor on $X=\Spec A$ with prime components $D_0,D_1$. We define a log scheme $X^+$ by endowing $X$ with the divisorial log structure induced by $D$. Assume that $X^+$ is log regular. Then $D_0$ and $D_1$ are $\Q$-Cartier, and $D_0\cdot D_1\leq 1$ with equality if and only if $A$ is regular.
\end{lemma}
\begin{proof}
We denote by $\mathcal{M}$ the multiplicative monoid consisting of the elements of $A$ that are invertible on $X\setminus D$, and we consider the characteristic monoid
$$\overline{\mathcal{M}}=\mathcal{M}/A^{\times}.$$ By the log regularity assumption, $D_0$ and $D_1$ are regular and $\overline{\mathcal{M}}$ is a toric monoid of dimension $2$.
  In particular, its groupification $\overline{\mathcal{M}}^{\mathrm{gp}}$ is a rank two lattice. The ring $A$ is regular if and only if the monoid $\overline{\mathcal{M}}$ is generated by two elements, that is, $\overline{\mathcal{M}}\cong \N^2$.

 Let $e_0$ and $e_1$ be the primitive generators of the one-dimensional faces of $\overline{\mathcal{M}}$. Then $e_0\wedge e_1$ generates $$m\cdot \bigwedge^2(\overline{\mathcal{M}}^{\mathrm{gp}})$$ for a unique positive integer $m$
 (in other words, $m$ is the absolute value of the determinant of $(e_0,e_1)$) and $m=1$ if and only if $A$ is regular.
 Since the fan of the log scheme $\Spec A$ is canonically isomorphic with $\Spec \overline{\mathcal{M}}$ by \cite[10.1]{kato-toric}, we know that (up to renumbering), $D_i$ is the zero locus in $\Spec A$ of
 the prime ideal $\overline{\mathcal{M}}\setminus \N e_{i}$ of $\overline{\mathcal{M}}$, for $i=0,1$ (by which we mean the zero locus of its inverse image in $\mathcal{M}$). Moreover, any representative $\widetilde{e}_i$ of $e_i$ in $\mathcal{M}$ is a regular local parameter on $D_{i}$, and the characteristic  monoid at the generic point of $D_i$ is $\overline{\mathcal{M}}/\N e_{1-i}$ (see the proof of \cite[4.3.2(1)]{ErHaNi} for a similar computation). It follows that $mD_i=\mathrm{div}(\widetilde{e}_{1-i})$, so that $D_1$ and $D_2$ are $\Q$-Cartier, and $mD_1\cdot D_2=1$. Thus $D_1\cdot D_2=1/m\leq 1$, with equality if and only if $A$ is regular.
\end{proof}

\begin{theorem}\label{thm:semiample}
Let $C$ be a $K$-curve of genus $g(C)\geq 1$ and let $\cC$ be its minimal $snc$-model.
 If $m$ is a sufficiently divisible positive integer, then the base locus of the line bundle $\omega_{\cC/R}(\cC_{k,\red})^{\otimes m}$ on $\cC$ is the union of the inessential components of $\cC_{k}$.
\end{theorem}
\begin{proof}
  Let $m$ be a positive integer. Using adjunction, one sees that the line bundle $\omega_{\cC/R}(\cC_{k,\red})^{\otimes m}$ has negative degree, resp.~degree $0$, on each rational curve in $\cC_k$ that intersects the other components in precisely one point, resp.~two points.
 It follows at once that the union of the inessential components in $\cC_k$ is contained in the base locus of
 $\omega_{\cC/R}(\cC_{k,\red})^{\otimes m}$. We will show there are no other points in the base locus if $m$ is sufficiently divisible. If $\cC_{k,\red}$ is either an elliptic curve or a loop of rational curves, then  $C$ has genus one and $\omega_{\cC/R}(\cC_{k,\red})^{\otimes m}$ is trivial for some $m>0$ by \cite[5.7 and 6.6]{LLR}. Thus we can discard these cases in the remainder of the proof.

  We can choose a reduced divisor $H$ on $\cC$ with the following properties:
  \begin{itemize}
  \item the divisor $H$ does not contain any prime component of $\cC_k$ (in other words, $H$ is horizontal) and $H+\cC_k$ is a divisor with strict normal crossings;
\item we have $H\cdot E=1$ if $E$ is a prime component of $\cC_k$ that corresponds to the end point of a maximal tail in $\Sk(\cC)$, and $H\cdot E=0$ for every other prime component of $\cC_k$.
    \end{itemize}
 We denote by $S^+$ the scheme $S=\Spec R$ endowed with its standard log structure (the divisorial log structure induced by the closed point of $S$) and by $\cC^+$ the log scheme we obtain by endowing $\cC$ with the divisorial log structure associated with the divisor $\cC_k+H$. Then $\cC^+$ is log regular in the sense of \cite{kato-toric} because $\cC_k+H$ has strict normal crossings.

 By Lipman's generalization of Artin's contractibility criterion \cite[27.1]{lipman}, any chain of rational curves in $\cC_k$ can be contracted to a rational singularity.
     In particular, there exists a morphism  $h:\cC\to \cD$ of normal proper $R$-models of $C$ that contracts precisely
  the rational components of $\cC_k$ that meet the rest of the special fiber in exactly one or two points.
  We endow $\cD$ with the divisorial log structure associated with $\cD_k+h_*H$ and denote the resulting log scheme by $\cD^+$. It follows from
\cite[\S3]{schroer} that $\cD^+$ is still log regular (this is the reason why we added the horizontal divisor $H$). The morphism $h$ induces a morphism of log schemes $h:\cC^+\to \cD^+$, and this morphism is log \'etale since it is a composition of log blow-ups.

 We consider the canonical line bundle $$\omega_{\cC^+/S^+}=\det \Omega^1_{\cC^+/S^+}$$ on $\cC$. It follows easily from \cite[3.3.4]{ErHaNi} that
$\omega_{\cC^+/S^+}$ is isomorphic to $\omega_{\cC/R}(\cC_{k,\red}+H)$.
  We can copy the proofs of \cite[3.3.2 and 3.3.6]{ErHaNi} to show that
the coherent sheaf $\Omega^1_{\cD^+/S^+}$ on $\cD$ is perfect, so that we can define the canonical line bundle
$$\omega_{\cD^+/S^+}=\det \Omega^1_{\cD^+/S^+}$$ on $\cD$ (the results in \cite{ErHaNi} were formulated for $H=0$ but the arguments carry over immediately). Since $h$ is log \'etale, \cite[3.3.6]{ErHaNi} also implies that we have a canonical isomorphism
$\omega_{\cC^+/S^+}\cong h^*\omega_{\cD^+/S^+}$.

By Lemma \eqref{lemm:factorial}, the divisor $h_*H$ is $\Q$-Cartier. Thus by choosing $m$ sufficiently divisible, we can assume that $mh_*H$ is Cartier. We will prove that the line bundle $\omega^{\otimes m}_{\cD^+/S^+}(-mh_*H)$ on $\cD$ is ample. This implies that its pullback to $\cC$ is semi-ample (that is, some tensor power is generated by its global sections). But this pullback is isomorphic to $\omega^{\otimes m}_{\cC^+/S^+}(-h^*h_*mH)$, which is a subbundle of
  $$\omega^{\otimes m}_{\cC^+/S^+}(-mH)\cong \omega_{\cC/R}(\cC_{k,\red})^{\otimes m}$$ that coincides with $\omega_{\cC/R}(\cC_{k,\red})^{\otimes m}$ away from the inessential components of $\cC_k$ (note that, for every closed point $x$ of $h_*H$, the inverse image $h^{-1}(x)$ is a maximal tail of inessential components in $\cC_k$).

Thus it is enough to show that $\omega^{\otimes m}_{\cD^+/S^+}(-mh_*H)$ is ample. By \cite[7.5.5]{liu}, it suffices to show that it has positive degree on every prime component $E$ of $\cD_k$. By adjunction, the restriction of $\omega_{\cD^+/S^+}$ to $E$ is isomorphic to $\omega_{E/k}(F)$ where $F$ is the reduced divisor on $E$ supported on the intersection points of $E$ with the other components of $\cD_k+h_*H$. Note that either $E$ has positive genus, or $F$ consists of at least three points
 including at least two intersections points of $E$ with the other components of $\cD_k$, since we contracted all the other components in $\cC_k$. Therefore, we only need to show that $h_*H_0\cdot E<1$ for every prime component $H_0$ of $H$. This follows from Lemma \ref{lemm:factorial} (note that $\cD$ is singular at every point of $h_*H\cap \cD_k$ by minimality of $\cC$).
\end{proof}
\begin{remark} In the language of \cite{NiXu}, the proof of Theorem \ref{thm:semiample} can also be interpreted as follows: the model for $C$ we obtain from the minimal $snc$-model $\cC$ by contracting all the inessential components in the special fiber is the minimal $dlt$-model of $C$ .
\end{remark}

  If $C$ has semi-stable reduction, we can be more precise: we will show that the logarithmic $2$-canonical  line bundle on the minimal $snc$-model of $C$ is generated by global sections. First, we prove two elementary lemmas.

\begin{lemma}\label{lemm:multvanish}
Let $\cX$ be a regular flat proper $R$-scheme of relative dimension one and let $\mathcal{L}$ be a line bundle on $\cX$. Let $E$ be an irreducible component of multiplicity $N$ in $\cX_k$ and let $a$ be an integer in $\{1,\ldots,N\}$.
 If the restriction of $\mathcal{L}((1-a)E)$ to $E$ has negative degree, then $$H^0(aE,\mathcal{L}|_{aE})=0.$$
\end{lemma}
\begin{proof}
 We prove this by induction on $a$. The case $a=1$ is obvious. Assume that $a>1$ and that the property holds for $a-1$.
  If $\mathcal{L}((1-a)E)$ has negative degree on $E$ then the same holds for $\mathcal{L}((b-a)E)$ for all $b\geq 1$ because $E^2\leq 0$.
  We consider the short exact sequence
 $$0\to \mathcal{L}|_{aE}\otimes \mathcal{I}\to \mathcal{L}|_{aE} \to \mathcal{L}|_{(a-1)E}\to 0$$ where $\mathcal{I}$ is the ideal sheaf of $(a-1)E$
 in $aE$. By our induction hypothesis, it suffices to show that
 $$H^0(aE,\mathcal{L}|_{aE}\otimes \mathcal{I})=0.$$ This follows from the isomorphism of $\mathcal{O}_{aE}$-modules
 $$\mathcal{L}|_{aE}\otimes \mathcal{I}\cong \mathcal{L}((1-a)E)|_{E}$$
 on $E$.
\end{proof}

\begin{lemma}\label{lemm:nonposdeg}
Let $\cX$ be a regular flat proper $R$-scheme of relative dimension one and let $\mathcal{L}$ be a line bundle on $\cX$.
 Let $D$ be a reduced connected divisor supported on $\cX_k$. Suppose that the restriction of $\mathcal{L}$ to each component in $D$ has non-positive degree, and that this degree is negative for at least one component. Then $$H^0(D,\mathcal{L}|_D)=0.$$
\end{lemma}
\begin{proof}
This follows easily by induction on the number $r$ of irreducible components of $D$. If $r=1$ the result is obvious. Suppose that $r>1$ and let $E$ be a component of $D$ on which $\mathcal{L}$ has negative degree. Then every section of $\mathcal{L}$ on $D$ vanishes on $E$, so that it is also a section of $\mathcal{L}(-E)$ on $D$. The line bundle
 $\mathcal{L}(-E)$ has negative degree on each irreducible component of $D$ that intersects $E$, so that we can apply the induction hypothesis to this line bundle and to every connected component of $D-E$.
\end{proof}

\begin{theorem}\label{thm:sstable}
 Let $C$ be a $K$-curve of genus $g(C)\geq 1$ with semi-stable reduction and let $\cC$ be its minimal $snc$-model.
 Then the logarithmic $2$-canonical line bundle $\omega_{\cC/R}(\cC_{k,\red})^{\otimes 2}$ on $\cC$ is generated
 by its global sections.
\end{theorem}
\begin{proof}
 It will be convenient to start from the minimal $nc$-model $\cC'$ of $C$ instead of the minimal $snc$-model $\cC$. We will show that
 $\omega_{\cC'/R}^{\otimes 2}$ is generated by global sections. This implies the desired result: the line bundle $\omega_{\cC/R}(\cC_{k,\red})$ is isomorphic to the pullback of $\omega_{\cC'/R}\cong \omega_{\cC'/R}(\cC'_k)$ through the morphism $g:\cC'\to \cC$, since $\cC'_k$ is reduced and $\cC$ is a composition of log blow-ups if we endow both
 models with the divisorial log structure associated with their special fibers.
  We can assume that $C$ has genus at least $2$, since otherwise, $\omega_{\cC'/R}$ is trivial.

 Let $x$ be a closed point of $\cC'_k$ and denote by $h:\cD\to \cC'$ the blow-up of $\cC'$ at $x$ and by $E_0$ the exceptional curve of $h$.
     Then $\omega_{\cC'/R}^{\otimes 2}$ is globally generated at $x$ if and only if the morphism
 $$H^0(\cD,h^*\omega_{\cC'/R}^{\otimes 2})\to H^0(E_0,h^*\omega_{\cC'/R}^{\otimes 2}|_{E_0})$$ is surjective. To prove this surjectivity property, it suffices to show that
$H^1(\cD,h^*\omega_{\cC'/R}^{\otimes 2}(-E_0))$ vanishes. By Serre duality, this is equivalent to showing that $H^0(\cD_k,\mathcal{L})=0$,  with
$$\mathcal{L}=(\omega_{\cD/R}\otimes (h^*\omega_{\cC'/R}^{-2})(E_0))|_{\cD_k}\cong (\omega^{-1}_{\cD/R}(3E_0))|_{\cD_k}.$$
We write $$\cD_k=N_0E_0+\sum_{i=1}^r E_i,$$ where $N_0$ is either one or two, depending on whether $x$ is a regular or singular point of $\cC'_k$.

 We first observe that the restriction of $\mathcal{L}$ to $N_0E_0$ has no non-zero global sections, because the restriction of the line bundle $\mathcal{L}((1-N_0)E_0)$ to $E_0$ has negative degree so that we can apply Lemma \ref{lemm:multvanish}.
  Thus every section of $\mathcal{L}$ on $\cD_k$ is also a section of $$\mathcal{L}'=(\omega^{-1}_{\cD/R}((3-N_0)E_0))|_{\cD_k}.$$ Note that $\mathcal{L}'$ has degree $-1$ on $E_0$ if $N_0=1$ and degree $0$ if $N_0=2$.
 Next, we consider any component $E_i\neq E_0$ in $\cD_k$. By the adjunction formula, the degree of $\mathcal{L}'$ on $E_i$ is given by
 \begin{equation}\label{eq:degree}
  \mathrm{deg}(\mathcal{L}'|_{E_i})=2-2p_a(E_i)+E_i^2+(3-N_0)E_0\cdot E_i.
  \end{equation}
  By the projection formula, $E_i^2=h(E_i)^2-\delta$, where
 \begin{itemize}
  \item$\delta=0$ if $x$ does not lie on $h(E_i)$,
   \item $\delta=1$ if $x$ is a regular point of $h(E_i)$ (then $E_0\cdot E_i=1$),
   \item $\delta=4$ if $x$ is a self-intersection point of $h(E_i)$ (then $N_0=2$ and $E_i\cdot E_0=2$).
   \end{itemize}
   Thus if
  $E_i^2+(3-N_0)E_0\cdot E_i$ is positive, we have  $\delta=1$ and $h(E_i)^2=0$, which means that $\cC'_k=h(E_i)$ and $p_a(E_i)=p_a(h(E_i))\geq 2$ by our assumption on the genus of $C$. Note also that $E_i^2+(3-N_0)E_0\cdot E_i= 0$ implies that $\delta=1$ and $h(E_i)^ 2=-1$, and thus $p_a(E_i)=p_a(h(E_i))>0$ since otherwise, $h(E_i)$ would be an exceptional curve on $\cC'$, contradicting minimality.

   It follows that the number in \eqref{eq:degree} is negative, unless:
 \begin{itemize}
 \item $p_a(E_i)=1$ and $E_i^2+(3-N_0)E_0\cdot E_i= 0$, or
  \item $p_a(E_i)=0$ and $E_i^2+(3-N_0)E_0\cdot E_i\in \{-2,-1\}$.
  \end{itemize}
If $p_a(E_i)=1$ and $E_i^2+(3-N_0)E_0\cdot E_i= 0$ then $h(E_i)$ is a $(-1)$-curve of arithmetic genus one and  $x$ is a point on $h(E_i)$ that does not lie on any other component of $\cC'_k$.
  Similarly, if $p_a(E_i)=0$ and $E_i^2+(3-N_0)E_0\cdot E_i=-1$ then $h(E_i)$ must be a regular rational $(-2)$-curve and $x$ is a point on $h(E_i)$ that does not lie on any other component of $\cC'_k$. Finally, if $p_a(E_i)=0$ and $E_i^2+(3-N_0)E_0\cdot E_i=-2$ then $h(E_i)$ contains $x$ or $h(E_i)$ is a regular rational curve of self-intersection number $-2$.

  From these observations, we can deduce the following properties:
  \begin{itemize}
  \item The divisor $\cD_k$ contains at most one component $E_i$ on which $\mathcal{L}'$ has positive degree.
   In that case, this degree equals one, and $h(E_i)$ is a regular rational $(-2)$-curve and it is the only component of $\cC'_k$ that contains $x$.
   Then each connected component of $\cD_k-N_0E_0-E_i$ contains a curve on which $\mathcal{L}'$ has negative degree, since such a component cannot consist entirely of regular rational $(-2)$-curves. It follows from Lemma \ref{lemm:nonposdeg} that $\mathcal{L}'$ has no non-zero global sections on $\cD_k-N_0E_0-E_i$. Then $\mathcal{L}$ has no non-zero global sections on $\cD_k$, because every section vanishes at the two intersection points of $E_i$ with $\cD_k-N_0E_0$.

  \item Assume that $\mathcal{L}'$ has non-positive degree on every component of $\cD_k$. The divisor $\cD_k$ contains at least one component $E_i$ on which $\mathcal{L}'$ has negative degree: if $x$ lies on only one component of $\cC'_k$ then we can take $E_i=E_0$. In the other case, all components of $\cD_k$ on which $\mathcal{L}'$ has degree zero are regular rational curves that intersect the rest of $\cD_k$ in precisely two points, and $\cD_k$ cannot consist entirely of such curves because of our assumption that $g(C)\geq 2$. Thus Lemma \ref{lemm:nonposdeg} again implies that $\mathcal{L}'$ has no non-zero global sections on $\cD_{k,\red}$, so that
      $H^0(\cD_k,\mathcal{L})=0$.
  \end{itemize}
This concludes the proof.
 \end{proof}

\begin{remark}
Theorem \ref{thm:sstable} also follows from Theorem 7 in \cite{lee}, which states that $\omega^{\otimes m}_{\cC_{\min}/R}$ is generated by global sections if $m\geq 2$ and $\cC_{\min}$ is the minimal regular model of a curve of genus $g\geq 2$ (recall that the minimal regular model of a curve with semi-stable reduction coincides with its minimal $nc$-model). Although we have only dealt with the semi-stable case, we feel that our proof of Theorem \ref{thm:sstable} is still interesting, because it uses a different method and it is substantially simpler than the proof of the more general result in \cite{lee}. It does not seem possible to deduce Theorem \ref{thm:semiample} from the semi-ampleness of
 $\omega_{\cC_{\min}/R}$ in a direct way, because of the discrepancy between the minimal regular model and the minimal $nc$-model of $C$ if $C$ does not have semi-stable reduction.
\end{remark}

\begin{theorem}\label{thm:essentialsk}
Let $C$ be a $K$-curve of genus $g(C)\geq 1$ and let $\cC$ be its minimal $snc$-model.
\begin{enumerate}
\item \label{it:esssk} The essential skeleton $\Sk(C)$ is equal to the combinatorial skeleton of $\Sk(\cC)$ (as a subspace of $C^{\an}$). In particular, $\Sk(C)$ is a strong deformation retract of $C^{\an}$.
\item \label{it:esssk-sst} If $C$ has semi-stable reduction, then $\Sk(C)=\Sk(\cC)$. Moreover,
 $$\Sk(C)=\bigcup_{\omega} \Sk(C,\omega)$$ where $\omega$ runs through the set of non-zero regular $2$-canonical forms on $C$.
\end{enumerate}
\end{theorem}
\begin{proof}
\eqref{it:esssk} It follows from Corollary \ref{cor:bound} that, for every non-zero regular pluricanonical form $\omega$ on $C$, the weight function $\wt_{\omega}$ is strictly increasing along every tail of $\Sk(\cC)$ if we orient the tail from its starting point to its end point. Thus the essential skeleton $\Sk(C)$ is contained in the combinatorial skeleton of $\Sk(\cC)$. The converse inclusion is a consequence of Theorem \ref{thm:semiample}: we choose a positive integer $m$ such that the base locus of  $\omega_{\cC/R}(\cC_{k,\red})^{\otimes m}$ is the union of  inessential components of $\cC_k$. If $x$ is a singular point of $\cC_{k,\red}$ that does not lie on an inessential component and $\omega$ is a global section of $\omega_{\cC/R}(\cC_{k,\red})^{\otimes m}$ that does not vanish at $x$, then the weight function $\wt_{\omega}$ vanishes on the edge of $\Sk(\cC)$ and it is non-negative on the whole skeleton $\Sk(\cC)$, so that the edge belongs to $\Sk(C,\omega)$. Thus the combinatorial skeleton of $\Sk(\cC)$ is equal to $$\bigcup_{\omega} \Sk(C,\omega)$$ where $\omega$ runs through any basis of the $R$-module
$$H^0(\cC,\omega_{\cC/R}(\cC_{k,\red})^{\otimes m}).$$

\eqref{it:esssk-sst} As we have already observed in \eqref{sss:tails}, the special fiber of $\cC$ does not contain any inessential components. Therefore, the combinatorial skeleton of $\Sk(\cC)$ is equal to $\Sk(\cC)$ and thus also to the essential skeleton $\Sk(C)$ by point \eqref{it:esssk}. The proof of \eqref{it:esssk}, together with Theorem \ref{thm:sstable}, shows that $2$-canonical forms $\omega$ suffices to generate the whole essential skeleton $\Sk(C)$.
\end{proof}

\subsection{The subset of the essential skeleton cut out by canonical forms}

\sss Let $C$ be a $K$-curve of genus $g\geq 1$ and denote by $\cC$ its minimal $snc$-model. We assume that $\cC_k$ is reduced.
 Looking at the definition of the essential skeleton in \eqref{sss:KSsk0}, it is natural to ask which part of the essential skeleton we recover by
 taking the union of the Kontsevich-Soibelman skeleta $\Sk(C,\omega)$ where $\omega$ runs through the set of non-zero canonical (rather than pluricanonical) forms on $C$.
 In this section, we will show that one obtains the union of all the closed {\em non-bridge edges} and all the vertices of positive genus of the skeleton $\Sk(\cC)$.
 Recall that a {\em bridge} in a graph $G$ is an edge that is not contained in any non-trivial cycle, or equivalently, that is contained in every spanning tree.

\sss Let $G=G(\cC_k)$ be the dual graph of the special fiber $\cC_k$, and let $\nu:\widetilde{\cC}_k\to \cC_k$
  be a normalization morphism.
 The set $V(G)$ of vertices $v$ of $G$ is in bijection with the set of connected components $C_v$ of  $\widetilde{\cC}_k$.
Let $E(G)$ denote the set of edges of $G$, each endowed with a fixed (but arbitrary) orientation. If $x$ is the singular point of $\cC_{k}$ corresponding to an edge $e$, then the choice of an orientation on $e$ amounts to choosing a point $y$ in $\nu^{-1}(x)$: if the oriented edge $\vec{e}$ points towards the vertex $v$, then we take $y$ to be the unique point of  $\nu^{-1}(x)$ lying on $C_v$.

\sss \label{sss:rosenlicht} By cohomological flatness of $\cC\to \Spec(R)$ and Grothendieck-Serre duality, the module $H^1(\cC,\omega_{\cC/R})$ is free, so that $$H^0(\cC_k,\omega_{\cC / R}) \otimes k \cong H^0(\cC_k,\omega_{\cC_k / k}).$$
 We can identify $H^0(\cC_k,\omega_{\cC_k / k})$  with the space of {\em Rosenlicht differentials} on $\cC_k$.
A Rosenlicht differential $\omega$ is, by definition, the data of a meromorphic differential $\omega_v$ on $C_v$ for each $v \in V(G)$ such that:
\begin{enumerate}
\item Each $\omega_v$ has at worst logarithmic poles at the inverse images under $\nu$ of the singular points of $\cC_k$, and it is regular everywhere else.
\item If  $x$ is a singular point of $\cC_k$ and $\nu^{-1}(x)=\{y_1,y_2\}$, then the residues of $\omega$ at $y_1$ and $y_2$ sum to zero.
\end{enumerate}
 Given $\omega \in H^0(\cC_k,\omega_{\cC_k / k})$ and an oriented edge $\vec{e} \in E(G)$, let ${\rm res}_{\vec{e}}(\omega)$ be the residue of $\omega$ at the point of $\widetilde{\cC}_k$ corresponding to $\vec{e}$.
By the residue theorem, the sum
$${\rm res}(\omega) := \sum_{e \in E(G)} {\rm res}_{\vec{e}}(\omega) (\vec{e})$$
belongs to $H_1(G,k)$, so that we obtain a morphism of $k$-vector spaces
$$\mathrm{res}:H^0(\cC_k,\omega_{\cC_k / k}) \to H_1(G,k)$$ which is called the residue map.

\begin{lemma} \label{lem:residue_ses}
The residue map fits into a short exact sequence of $k$-vector spaces
$$\minCDarrowwidth15pt\begin{CD} 0 @>>> \oplus_{v \in V(G)} H^0(C_v, \omega_{C_v / k}) @>\alpha>> H^0(\cC_k,\omega_{\cC_k / k}) @>\mathrm{res}>> H_1(G,k) @>>> 0.\end{CD}$$
\end{lemma}

\begin{proof}
By the definition of Rosenlicht differentials and the residue map, the kernel of $\mathrm{res}$ is equal to  $\bigoplus_{v \in V(G)} H^0(C_v, \omega_{C_v / k})$.
Surjectivity of the residue map now follows by a dimension count, since $$\dim_k H^0(\cC_k,\omega_{\cC_k / k})=\dim_k H_1(G,k) + \sum_{v \in V(G)} \dim_k H^0(C_v, \omega_{C_v / k})=g.$$
\end{proof}

\begin{lemma}\label{lem:onedge}
Let $\omega$ be a regular canonical form on $C$ and let $v$ be a vertex of genus zero of $\Sk(\cC)$. Then $v$ belongs to $\Sk(C,\omega)$ if and only if some edge adjacent to $v$ belongs to $\Sk(C,\omega)$.
\end{lemma}
\begin{proof} The ``if" part follows from the fact that $\Sk(C,\omega)$ is closed, so we only need to prove the converse implication.
 We denote by $f$ the restriction of $\wt_{\omega}$ to $\Sk(\cC)$. Assume that $v$ lies in $\Sk(C,\omega)$, that is, $f$ reaches its minimal value at $v$.
 By Corollary \ref{cor:bound}, the degree of $\Delta(f)$ at $v$ is strictly less than the valency of $v$ in $\Sk(\cC)$.
  Since $f$ has integer slopes, this means that at least one of the outgoing slopes of $f$ from $v$ must be zero, so that the corresponding edge also lies in $\Sk(C,\omega)$.
\end{proof}

\begin{theorem}\label{thm:nonbridgeessentialsk}
If $C$ is a $K$-curve of genus $g \geq 1$ whose minimal $snc$-model $\cC$ over $R$ is semi-stable, then the union $S=\bigcup_{\omega} \Sk(C,\omega)$, as $\omega$ runs through the set of non-zero global sections of $\omega_{C/K}$, is equal to the union of all the closed non-bridge edges and all the vertices of positive genus of $\Sk(\cC)$.
\end{theorem}
\begin{proof}
Multiplying a non-zero canonical form $\omega$ with $a\in K^{\times}$ shifts the weight function $\wt_{\omega}$ by $v_K(a)$ and does not affect $\Sk(C,\omega)$. Moreover, since $\cC_k$ is reduced,
 $\wt_{\omega}$ takes integer values at the vertices of $\Sk(\cC)$. Thus in the definition of $S$, we only need to consider canonical forms $\omega$ whose minimal value on $\Sk(\cC)$ equals zero (recall from \eqref{sss:KSsk0} that this minimal value is always reached at a vertex). Now it is clear from the definition of the weight function that $\wt_{\omega}$ vanishes at an edge, resp.~vertex, of $\Sk(\cC)$, if and only if $\omega$ generates $\omega_{\cC/R}(\cC_{k})$ at the corresponding point of $\cC_k$, resp.~at the generic point of the corresponding irreducible component of $\cC_k$. Thus, in order to find the faces of $\Sk(\cC)$ that lie in $S$, we need to determine which singular points and irreducible components of $\cC_k$ lie in the base locus of $\omega_{\cC/R}(\cC_{k})\cong \omega_{\cC/R}$.
 For this aim, we can use Rosenlicht differentials: a point of $\cC_k$ lies in the base locus of $\omega_{\cC/R}$ if and only if it lies in the base locus of $\omega_{\cC_k/k}$ on $\cC_k$, by the surjectivity of  the reduction map $$H^0(\cC,\omega_{\cC / R}) \to H^0(\cC_k,\omega_{\cC_k/ k}).$$

 Using the morphism $\alpha$ in Lemma \ref{lem:residue_ses}, we can find an element of $H^0(\cC_k,\omega_{\cC_k/ k})$ that generates $\omega_{\cC_k/ k}$ at the generic point of every component of positive genus of $\cC_k$. In particular, all the vertices of positive genus of $\Sk(\cC)$ belong to $S$.
  By Lemma \ref{lem:onedge}, it now suffices to determine which edges of $\Sk(\cC)$ lie in $S$. The residue theorem immediately implies that a bridge never belongs to $S$, while the surjectivity of the residue map in Lemma \ref{lem:residue_ses} shows that every non-bridge edge lies in $S$. This concludes the proof.
\end{proof}
\begin{remark}
The essential skeleton $\Sk(C)$ is always connected, by Theorem \ref{thm:essentialsk}, but it is easy to use the proof of Theorem \ref{thm:nonbridgeessentialsk} to produce examples of a curve $C$ and a non-zero canonical form $\omega$ such that $\Sk(C,\omega)$ is disconnected (for instance, when $\Sk(\cC)$ is a chain with vertices of positive genus).
\end{remark}

\subsection{An alternate approach to computing the essential skeleton of a maximally degenerate semi-stable curve}

\sss Let $C$ be a $K$-curve of genus $g\geq 1$ and denote by $\cC$ its minimal $snc$-model. There is an elegant way to prove Theorems~\ref{thm:essentialsk}(2) and \ref{thm:nonbridgeessentialsk} using potential theory on metric graphs if we assume that $C$ is a {\em maximally degenerate} $K$-curve.
 This assumption is common in tropical geometry: it means that $\cC_k$ is reduced and that all the irreducible components of $\cC_k$ are rational curves. This implies that the metric graph $\Sk(\cC)$ still has genus $g$.
  The proofs yield some additional information about the structure of $\Sk(C,\omega)$ for certain explicit 2-canonical forms
$\omega$. They also have the advantage that they can be extended to the non-discretely valued setting (cf. Remark~\ref{rmk:nondiscretecase}).

 \sss For background on potential theory on metric graphs, see for instance \cite{baker-spec}. We recall that a {\em tropical rational function} on a metric graph $\Gamma$ is a real-valued continuous piecewise affine function on $\Gamma$ with integral slopes, and that the divisor of such a function is defined by $\mathrm{div}(f)=-\Delta(f)$. In other words, the degree of ${\rm div}(f)$ at a point of $\Gamma$ is the sum of the {\em incoming} slopes of $f$. Two divisors on $\Gamma$ are called {\em equivalent} if they differ by the divisor of a tropical rational function.
 We begin with a combinatorial lemma needed for our alternate proof of Theorem~\ref{thm:nonbridgeessentialsk}.


\begin{lemma} \label{lem:min_spanning_tree} Let $G$ be a discrete graph without loops and denote by $\Gamma$ the metric graph associated with $G$.
Let $T$ be a spanning tree of $\Gamma$, let $e$ be an edge of $\Gamma$ not contained in $T$, and let $Z(T,e)$ be the unique cycle in $T \cup e$.
Let $D$ be an effective divisor on $\Gamma$ which is equivalent to the canonical divisor $K_G$ and whose support contains a point $p_i$ from the relative interior of each edge $e_i \neq e$ contained in the complement of $T$.
Finally, let $f$ be a tropical rational function on $\Gamma$ with ${\rm div}(f) = D-K_G$. Then the locus of points $p \in \Gamma$ where $f$ achieves its minimum value is equal to $Z(T,e)$.
\end{lemma}

\begin{proof}
Let $\Sk(f)$ be the locus of $p \in \Gamma$ at which $f$ attains its minimum value.
For each $p \in \Sk(f)$, $f$ can be strictly increasing in at most ${\rm val}(p) - 2$ tangent directions, since it has slope at least $1$ in each such direction and non-negative slope in every other direction and the total sum of outgoing slopes of $f$ at $p$ is at most $$\mathrm{deg}_p K_G={\rm val}(p) - 2.$$
Thus there are at least two tangent directions at $p$ along which $f$ is constant.  It follows that every connected component of $\Sk(f)$ is a graph in which every vertex has valency at least $2$.
 However, $\Sk(f)$ cannot contain any of the points $p_i$, since the sum of the outgoing slopes of $f$ at $p_i$ is equal to $-\mathrm{deg}_{p_i}D<0$.
 Thus $\Sk(f) \subset T \cup e$, and the only possible cycle in $\Sk(f)$ is  $Z(T,e)$.
 Hence, $\Sk(f) = Z(T,e)$ as claimed.
\end{proof}

We obtain the following strengthening of Theorem~\ref{thm:nonbridgeessentialsk} in this context (it can also be deduced directly from Lemma \ref{lem:residue_ses} and the proof of Theorem \ref{thm:nonbridgeessentialsk}):

\begin{prop}\label{prop:pot1}
Assume that $C$ is maximally degenerate.
If $e$ is a non-bridge edge of $\Sk(\cC)$, then there exists a non-zero canonical form $\omega \in H^0(C,\omega_{C/K})$ such that $\Sk(C,\omega)$ is a simple cycle with $e$ in its support.
\end{prop}

\begin{proof}
Since $e$ is not a bridge, there exists a spanning tree $T$ of $\Sk(\cC)$ not containing $e$.  Let $e=e_0,e_1,\ldots,e_{g-1}$ be the edges of $\Sk(\cC)$ not contained in $T$, and
 choose a type II point $p_i$ in the relative interior of $e_i$ for every $i$ in $\{1,\ldots,g-1\}$ (type II points are the divisorial points in the terminology of \cite{MuNi}).
 We set $D_0= p_1+\cdots+p_{g-1}$. We would like to find a divisor $\widetilde{D}_0$ on $C$ such that $(\rho_{\cC})_*(\widetilde{D}_0)=D_0$. Unfortunately, this is not possible, since only the vertices of $\Sk(\cC)$ lift to $K$-rational points of $C$.

 This issue can be solved in the following way. Let $K'$ be a finite Galois extension of $K$ whose degree $n=[K':K]$ is not divisible by the characteristic of $k$. We denote by $R'$ the valuation ring of $K'$. Set $C'=C\times_K K'$ and let $\cC'$ be the minimal resolution of  $\cC\times_R R'$. Then it is well known, and easy to see, that $\cC'$ is the minimal $snc$-model of $C'$, and $\cC'_k$ is reduced. Moreover, the projection morphism $\pi:(C')^{\an}\to C^{\an}$ induces a homeomorphism $\Sk(\cC')=\pi^{-1}(\Sk(\cC))\to \Sk(\cC)$, and $\Sk(\cC')$ is obtained from $\Sk(\cC)$ by subdividing each edge into $n$ edges. Now we choose each point $p_i$ to be a vertex of $\Sk(\cC')$ in the relative interior of $e_i$. Then we can find a divisor $\widetilde{D}_0$ on $C'$ such that $(\rho_{\cC'})_*(\widetilde{D}_0)=D_0$.

 Since $H^0(C',\omega_{C'/K'})$ has dimension $g$ and $\widetilde{D}_0$ has degree $g-1$, there exists a non-zero $\omega' \in H^0(C',\omega_{C'/ K'}(-\widetilde{D}_0))$.
 Let $f$ be the restriction of ${\rm wt}_{\omega'}$ to $\Sk(\cC')$.
By Corollary~\ref{cor:bound}, we have
$$ {\rm div}(f) = (\rho_{\cC'})_*({\rm div}_{C'}(\omega')) - K_{\Sk(\cC')}.$$
 If we set $D=(\rho_{\cC'})_* ({\rm div}_{C'}(\omega'))$, then $D \geq D_0$ by construction.
 Now it follows from Lemma~\ref{lem:min_spanning_tree} that $\Sk(C',\omega')=\Sk(f)$ is a simple cycle that contains $e$.

 It remains to produce a non-zero element $\omega$ of $H^0(C,\omega_{C/K})$ such that $\Sk(C,\omega)=\Sk(C',\omega')$.
  Multiplying $\omega'$ with a suitable element of $(K')^{\times}$, we can assume that the minimal value of $\wt_{\omega'}$ on $\Sk(\cC')$ is equal to $0$.
  We denote by $\omega\in H^0(C,\omega_{C/K})$ the trace of $\omega'$ with respect to the Galois extension $K'/K$. Then it is easy to see 
   that $\wt_{\omega}=\wt_{\omega\otimes_K K'}\geq \wt_{\omega'}$ on $\Sk(\cC)$. 
  It is also clear that every singular point $x$ of $\cC'_k$ is fixed under the action of $\mathrm{Gal}(K'/K)$. 
  Thus the logarithmic residues at $x$ of the conjugates of $\omega'$ are all equal, and their sum is non-zero if and only if the logarithmic residue of $\omega'$ at $x$ is non-zero. It follows that an edge of $\Sk(\cC')$ lies in the zero locus of $\wt_{\omega'}$ if and only if it lies in the zero locus of $\wt_{\omega}$. Since $\Sk(C',\omega')$ is a union of edges, it follows that $\Sk(C,\omega)=\Sk(C',\omega')$.
\end{proof}

We now show that if $e$ is a bridge edge of $\Sk(\cC)$, then there exists a $2$-canonical form $\omega$ such that $\Sk(C,\omega)$ contains $e$, providing a new proof of Theorem~\ref{thm:essentialsk}(2) in the present context.

\begin{lemma} \label{lem:bridge_construction} Let $G$ be a discrete graph without loops and denote by $\Gamma$ the metric graph associated with $G$.
 We assume that $G$ has  no $1$-valent vertices.
 Choose any maximal chain $B$ of bridge edges  in $\Gamma$. We denote $v_1,v_2$ the endpoints of $B$.
 Let $T$ be a spanning tree in $\Gamma$. 
 Let $D$ be an effective divisor on $\Gamma$ equivalent to $2K_G$ satisfying the following properties:
\begin{enumerate}
\item\label{it:cond1} the support of $D$ contains a point from the relative interior of each edge contained in the complement of $T$;
\item\label{it:cond2} $D \geq K_G - (v_1)-(v_2)$.
\end{enumerate}
Finally, let $f$ be a tropical rational function on $\Gamma$ with ${\rm div}(f) = D - 2K_G$. Then the locus of points $p \in \Gamma$ where $f$ achieves its minimum value is equal to $B$.
\end{lemma}

\begin{proof}
 Let $\Sk(f)$ be the locus of $p \in \Gamma$ at which $f$ attains its minimum value. We can argue in the same way as in the proof of Lemma~\ref{lem:min_spanning_tree}.
By condition \eqref{it:cond2}, for each $p \neq v_1,v_2$ in $\Sk(f)$ there are at least two tangent directions at $p$ along which $f$ is constant, and if $p\in \{v_1,v_2\}$ there is at least one such direction. Thus every connected component of $\Sk(f)$
 is a graph in which every vertex different from $v_1,v_2$ has valency at least two, and it cannot be equal to $\{v_1\}$ or $\{v_2\}$. On the other hand, by condition \eqref{it:cond1}, the set $\Sk(f)$ cannot contain any cycles. It follows that $\Sk(f)=B$.
\if false
Let $\Sk(f)$ be the locus of $p \in \Gamma$ at which $f$ attains its minimum value, and let $S_2$ be the locus of $p \in \Gamma_2 \cup \{ e \}$ at which
the restriction of $f$ to $\Gamma_2 \cup \{ e \}$ attains its minimum value.
By condition (2) and the proof of Lemma~\ref{lem:min_spanning_tree}, for each $p \neq v_1,v_2$ in $S_2$ there are at least two tangent directions at $p$ along which $f$ is constant, and if $p=v_2 \in S_2$ then there are at least three such directions.
Moreover condition (1) guarantees that $S_2$ contains a unique cycle $Z = Z(T,e_2)$.
Thus one of the following must hold:
\begin{enumerate}
\item[(a)] $S_2$ contains $Z$ and $f$ is constant at $v_2$ in the direction of $e$.
\item[(b)] $S_2 = \{ v_1 \}$ and $f$ is decreasing at $v_2$ in the direction of $e$.
\end{enumerate}

In addition, since $f$ has no poles along $e^\circ$, in case (a) we have either:
\begin{enumerate}
\item[(a1)] $f$ is constant along $e$; or
\item[(a2)] $f$ is increasing at $v_1$ in the direction of $e$
\end{enumerate}
and in case (b) $f$ must also be increasing at $v_1$ in the direction of $e$.

Let $S_1$ be the locus of $p \in \Gamma_1 \cup \{ e \}$ at which
the restriction of $f$ to $\Gamma_1 \cup \{ e \}$ attains its minimum value.
By conditions (1) and (2) and the proof of Lemma~\ref{lem:min_spanning_tree}, $S_1 \subset e$ and $f$ is constant in at least one tangent direction
at $v_1$.  The latter condition implies that $S_1 \neq \{ v_1 \}$ and thus (a2) and (b) cannot hold.
The only remaining possibility is (a1), in which case we have $S_1 = e$ and $S_2 = Z \cup e$.  Thus $\Sk(f) = Z \cup e$ as desired.
\fi
\end{proof}

\begin{prop}\label{prop:pot2}
 Assume that $C$ is maximally degenerate.
Let $B$ be any maximal chain of bridge edges of $\Sk(\cC)$. 
 Then there exists a non-zero $2$-canonical form $\omega \in H^0(C,\omega^{\otimes 2}_{C/K})$ such that $\Sk(C,\omega)=B$.
\end{prop}

\begin{proof} We can assume that $g\geq 2$ since in the genus one case, $\Sk(\cC)$ is a cycle and does not contain any bridges.
 Since $\cC$ is the minimal $snc$-model of $C$ and $\cC_k$ is reduced, $\Sk(\cC)$ has no $1$-valent vertices. We set $\Gamma=\Sk(\cC)$. We choose a spanning tree $T$ of $\Gamma$.
  We define $K'$, $C'$ and $\cC'$ as in the proof of Proposition \ref{prop:pot1}. Then, by the same arguments as in that proof, it suffices to find a non-zero element $\omega' \in H^0(C',\omega^{\otimes 2}_{C'/K'})$ such that $\Sk(C',\omega')=B$.
 We can find an effective divisor $\widetilde{D}_0$ on $C'$ of degree $3g-4=g+(2g-4)$ over $K'$ such that $D_0=(\rho_{\cC'})_*(\widetilde{D}_0)$ satisfies properties \eqref{it:cond1} and \eqref{it:cond2} from the statement of Lemma \ref{lem:bridge_construction}.
 Since the space $H^0(C',\omega^{\otimes 2}_{C' / K'})$ has dimension $3g-3$ by Riemann-Roch, there exists a
non-zero $2$-canonical form $\omega' \in H^0(C,\omega^{\otimes 2}_{C'/K'})$ with ${\rm div}_{C'}(\omega') \geq \widetilde{D}_0$.
 We set $D=(\rho_{\cC'})_*({\rm div}_{C'}(\omega'))$. Then $D\geq D_0$ by construction.
Let $f$ be the restriction of ${\rm wt}_{\omega'}$ to $\Gamma$.
By Theorem~\ref{thm:laplacian}, we have
$$ {\rm div}(f) = D - 2K_{\Gamma}.$$
 The result now follows from Lemma~\ref{lem:bridge_construction}.
\end{proof}

\section{Appendix: the stable metric on $\mathbb{H}_0(C)$}\label{sec:appendix}
\subsection{Definition of the stable metric}
\sss Let $C$ be a $K$-curve. The metric on $\mathbb{H}_0(C)$ defined in Theorem \ref{thm:metric} was well-suited for the description of the Laplacian of the weight function in Theorem \ref{thm:laplacian}, but it does not behave well under extensions of the base field $K$. We will now define an alternative metric on $\mathbb{H}_0(C)$, which we call the {\em stable} metric, which has better properties with respect to base change. In particular, if $k$ has characteristic zero, one can compare it to the skeletal metric from \cite{BPR} (see Proposition \ref{prop:skeletal-metric}).

\sss We first put a metric on the geometric realization $\Gamma$ of a weighted discrete graph $(G,w)$ by replacing the formula in
\eqref{def:metric} by
$$\ell(e)=\frac{1}{\lcm\{w(v_1),w(v_2)\}}.$$
 Now the same arguments as in Section \ref{ss:def-metric} show that this definition induces a unique metric
 on $\mathbb{H}_0(C)$ such that, for every $snc$-model $\cC$ of $C$, the embedding
 $$\Gamma(\cC_k)\to \mathbb{H}_0(C)$$ is an isometry onto $\Sk(\cC)$. We call this metric the  stable metric on $\mathbb{H}_0(C)$. Note that, if $\cC_k$ is reduced, the stable metric on $\Sk(\cC)$ coincides with the one defined in Theorem \ref{thm:metric}.

\sss \label{sss:Zaff}  By \cite[\S3.2]{MuNi}, the skeleton $\Sk(\cC)$ of an $nc$-model $\cC$ of $C$ carries a
natural $\Z$-affine structure. If $e$ is an edge of $\Sk(\cC)$
with endpoints $v_1$ and $v_2$, then a $\Z$-affine function
$$f:e\setminus\{v_1,v_2\}\to \R$$ is a function of the form
$$(x_1,x_2)\mapsto ax_1/N_1+bx_2/N_2+c$$ where $a,b,c$ are integers, $N_1=w(v_1)$, $N_2=w(v_2)$ and
$x_1$ and $x_2=1-x_1$ are barycentric coordinates on $e\setminus
\{v_1,v_2\}\cong \,]0,1[$ such that the limit of $x_1$ at $v_1$ is
$1$ and the limit of $x_2$ at $v_2$ is $1$ (beware that we are not
excluding the possibility $v_1=v_2$). This definition is motivated
by the following fact: if $h\neq 0$ is a rational function on $C$,
then
$$\Sk(\cC)\to \R:x\mapsto -\ln |h(x)|$$ is continuous and piecewise $\Z$-affine,
 and this function is affine on an edge $e$ if and only if the point of $\cC_k$ corresponding to $e$ does not belong to the horizontal part of the divisor $\mathrm{div}_{\cC}(h)$ on $\cC$ (see \cite[3.2.2]{MuNi}).
 Moreover, if $e$ is an edge of $\Sk(\cC)$ that is not a loop, then every $\Z$-affine function on $e\setminus
\{v_1,v_2\}$ can be written as $$x\mapsto -\ln |h(x)|$$
 for some rational function $h\neq 0$ on $C$ (simply consider a monomial with suitable integer exponents in the local
 equations for the components corresponding to the vertices adjacent to $e$).

\sss The $\Z$-affine structure on $\Sk(\cC)$ induces
the stable metric on $\Sk(\cC)=\Gamma(\cC_k)$, in the following sense: the length of
$e$ is equal to
$$\inf_{f}\{\,|\lim_{0} f-\lim_{1} f|\,\}$$
where $f$ runs through the set of injective $\Z$-affine functions
$$f:e\setminus\{v_1,v_2\}\to \R,$$
 and where $\lim_{i}f$ denotes the limit of $f$ at $i$ for $i=0,1$, where we choose any homeomorphism to identify  $e\setminus \{v_1,v_2\}$ with the open interval $]0,1[$.

To see this, note that this
infimum is equal to the smallest positive element of the set
$$\{a/N_1-b/N_2\,|\,a,b\in \Z\}$$ which is precisely
$$\frac{\gcd(N_1,N_2)}{N_1N_2}=\frac{1}{\lcm(N_1,N_2)}.$$ Thus our definition of the
length of $e$ is the unique one such that the affine functions on
$e\setminus \{v_1,v_2\}$ are precisely the differentiable
functions with constant integer slope whose value at $v_1$ is a
multiple of $1/N_1$.

\subsection{Comparison with the skeletal metric}
\sss The set $$\mathbb{H}_0(C \times_K \widehat{K^a})=(C \times_K \widehat{K^a})^{\an}\setminus \{\mbox{points of type I and IV}\}$$ carries a
natural metric, which was called the {\em skeletal metric} in
\cite{BPR}. Its construction is described in detail in
\cite[\S{5.3}]{BPR}. We will now compare it to the metric that we
defined on $\mathbb{H}_0(C)$, in the case where $k$ has characteristic zero.

\sss  Let $C$ be a $K$-curve and let $\cC$ be an $snc$-model for
$C$. An irreducible component of $\cC_k$ is called {\em principal} if it has positive genus or it is a rational curve that intersects the rest of $\cC_k$ in at least three points. A principal vertex of $\Sk(\cC)$ is a vertex corresponding to a principal component in $\cC_k$.

\begin{prop}\label{prop:skeletal-metric}
Assume that $k$ has characteristic zero.
 Let $C$ be a $K$-curve and let $\cC$ be an $snc$-model of $C$.
  Denote by $\pi$ the canonical projection $C\times_K \widehat{K^a}\to C$.
 Then the corestriction $$\pi_{\cC}:\pi^{-1}(\Sk(\cC))\to \Sk(\cC)$$ of $\pi$ to $\Sk(\cC)$ is a local isometry over the complement of the principal vertex set of $\Sk(\cC)$. Moreover, if $\cC$ is semi-stable, then $\pi_{\cC}$ is an isometry.
\end{prop}
\begin{proof}
This can be deduced in a rather straightforward way from the results in Sections 1 and 4 of Chapter 3 in
 \cite{HaNi-mem}. Since the arguments are somewhat tedious and the result is not needed in this paper, we omit the proof.
\end{proof}


\begin{thebibliography}{9999}

\bibitem[Ba08]{baker-spec}
M.~Baker.
\newblock Specialization of Linear Systems from Curves to Graphs.
\newblock {\em Algebra \& Number Theory} 2(6):613--653, 2008.

\bibitem[BPR13]{BPR}
M.~Baker, S.~Payne and J.~Rabinoff.
\newblock On the structure of non-Archimedean analytic curves.
\newblock In: {\em Tropical and non-Archimedean geometry.} Vol.~605 of {\em Contemp. Math.}, Amer. Math. Soc., Providence, RI, pages 93--121, 2013.

\bibitem[Be90]{berkbook}
V.~G. Berkovich.
\newblock {\em {Spectral theory and analytic geometry over non-archimedean
  fields}}. Volume~33 of {\em Mathematical Surveys and Monographs}.
\newblock American Mathematical Society, Providence, RI (1990).


\bibitem[EHN15]{ErHaNi}
D.~Eriksson, L.H.~Halle and J.~Nicaise.
\newblock A logarithmic interpretation of Edixhoven's jumps for Jacobians.
\newblock {\em Adv.~Math.}, 279:532--574, 2015.

\bibitem[GRW14]{GubRabWer}
W.~Gubler, J.~Rabinoff and A.~Werner.
\newblock Skeletons and tropicalizations.
\newblock {\em Preprint}, arXiv:1404.7044.

\bibitem[HN14]{HaNi-mem}
L.H.~Halle and J.~Nicaise.
\newblock{N\'eron models and base change.}
\newblock To appaer in {\em Lecture Notes in Mathematics},  arXiv:1209.5556.

\bibitem[IS14]{schroer}
H.~Ito and S.~Schr{\"o}er.
\newblock Wild quotient surface singularities whose dual graphs are not star-shaped.
\newblock Preprint, arXiv:1209.3605.

\bibitem[Jo15]{jonsson}
M.~Jonsson.
\newblock Dynamics on Berkovich spaces in low dimensions.
\newblock In: A.~Ducros, C.~Favre and J.~Nicaise (editors), {\em Berkovich spaces and applications.} Volume~2119 of {\em Lecture Notes in Mathematics}. Springer, Cham, pages 205--366, 2015.

\bibitem[Ka89]{kato-log}
K.~Kato. \newblock Logarithmic structures of Fontaine-Illusie.
\newblock In: {\em Algebraic analysis, geometry, and number theory}. Johns Hopkins Univ. Press, Baltimore, MD, pages 191--224, 1989.

\bibitem[Ka94]{kato-toric}
K.~Kato.
\newblock{Toric singularities.}
\newblock {\em Am. J. Math.}, 116(5):1073--1099, 1994.

\bibitem[KS04]{kato-saito}
K.~Kato and T.~Saito.
\newblock On the conductor formula of Bloch.
\newblock {\em Publ. Math. Inst. Hautes \'Etudes Sci.}, 100:5--151, 2004.

\bibitem[Le05]{lee}
J.~Lee.
\newblock Relative canonical sheaves of a family of curves.
\newblock {\em J. Algebra}, 286(2)341--360, 2005.

\bibitem[Li69]{lipman}
J. ~Lipman.
\newblock Rational singularities, with applications to algebraic surfaces and unique factorization.
\newblock {\em Publ. Math. Inst. Hautes \'Etudes Sci.},  36:195--279, 1969.

\bibitem[Li02]{liu}
Q.~Liu.
\newblock {\em Algebraic geometry and arithmetic curves}. Volume~6 of {\em
  Oxford Graduate Texts in Mathematics}.
\newblock Oxford University Press, Oxford, 2002.

\bibitem[LLR04]{LLR}
Q.~Liu, D.~Lorenzini, and M.~Raynaud.
\newblock {N\'eron models, Lie algebras, and reduction of curves of genus one.}
\newblock {\em Invent. Math.}, 157(3):455--518, 2004.


\bibitem[MN15]{MuNi}
M.~Musta\c{t}\u{a} and J.~Nicaise.
 \newblock Weight functions on non-archimedean analytic spaces and the
Kontsevich-Soibelman skeleton. \newblock {\em Alg. Geom.}, 2(3):365--404, 2015.


\bibitem[Ni13]{Ni-tameram}
J.~Nicaise.
\newblock {Geometric criteria for tame ramification.}
\newblock {\em Math. Z.}, 273(3):839--868, 2013.

\bibitem[NX13]{NiXu}
J.~Nicaise and C.~Xu.
\newblock The essential skeleton of a degeneration of algebraic varieties.
\newblock   {\em To appear in Amer. J. Math.},  arXiv:1307.4041.

\bibitem[Tem14]{Temkin}
M.~Temkin.
\newblock Metrization of differential pluriforms on Berkovich analytic spaces.
\newblock   {\em Preprint},  arXiv:1410.3079.
\end{thebibliography}
\end{document}